\documentclass[11pt]{article}

\usepackage{epsfig,amsmath,amssymb,amsthm,graphics,verbatim,amsfonts,subfigure,psfrag}
\usepackage[letterpaper,margin=1in]{geometry}

\newtheorem{thm}{Theorem}[section]

\newtheorem{prop}[thm]{Proposition}
\newtheorem{lem}[thm]{Lemma}

\def\R{\mathbb{R}}

\def\N{\mathbb{N}}
\def\I{\infty}

\newcommand{\be}{\begin{equation}}
\newcommand{\ee}{\end{equation}}
\newcommand{\bea}{\begin{eqnarray}}
\newcommand{\eea}{\end{eqnarray}}
\newcommand{\beann}{\begin{eqnarray*}}
\newcommand{\eeann}{\end{eqnarray*}}
\newcommand{\benn}{\begin{equation*}}
\newcommand{\eenn}{\end{equation*}}

\def\ra{\rightarrow}
\def\I{\infty}

\newcommand{\cB}{{\mathcal B}}  % calligraphic B
\newcommand{\cC}{{\mathcal C}}  % calligraphic C
\newcommand{\cH}{{\mathcal H}}  % calligraphic H
\newcommand{\cM}{{\mathcal M}}  % calligraphic M
\newcommand{\cN}{{\mathcal N}}  % calligraphic N
\newcommand{\cO}{{\mathcal O}}  % calligraphic O
\newcommand{\cS}{{\mathcal S}}  % calligraphic S
\begin{document}

\author{Christian Kuehn\thanks{Institute for Analysis and Scientific Computing, Vienna University of Technology, Vienna, 1040, Austria}}
 
\title{Normal Hyperbolicity and Unbounded Critical Manifolds}

\maketitle

\begin{abstract}
This work is motivated by mathematical questions arising in differential equation models for autocatalytic reactions. In particular, this paper answers an open question posed by Guckenheimer and Scheper [SIAM J. Appl. Dyn. Syst. 10-1 (2011), pp. 92-128] and provides a more general theoretical approach to parts of the work by Gucwa and Szmolyan [Discr. Cont. Dyn. Sys.-S. 2-4 (2009), pp. 783-806]. We extend the local theory of singularities in fast-slow polynomial vector fields to classes of unbounded manifolds which lose normal hyperbolicity due to an alignment of the tangent and normal bundles. A projective transformation is used to localize the unbounded problem. Then the blow-up method is employed to characterize the loss of normal hyperbolicity for the transformed slow manifolds. Our analysis yields a rigorous scaling law for all unbounded manifolds which exhibit a power-law decay for the alignment with a fast subsystem domain. Furthermore, the proof also provides a technical extension of the blow-up 
method itself by augmenting the analysis with an optimality criterion for the blow-up exponents.
\end{abstract}


\section{Introduction}
\label{sec:introduction} 

The motivation of this work are several models of singularly perturbed differential equations arising in applications.
In \cite{MerkinNeedhamScott} Merkin et {al.} propose to study a prototypical autocatalytic system given by four reactions 
\be
\label{eq:reaction_2D}
P\ra Y,\quad Y\ra X,\quad  Y+2X\ra 3X, \quad X\ra Z
\ee 
where $X,Y$ are the two main reactants, $P$ is a constant 'pool'-chemical and $Z$ is the product. Then it can be shown, using standard mass-action kinetics and non-dimensionalization \cite{PetrovScottShowalter}, that \eqref{eq:reaction_2D} leads to a two-dimensional system of ordinary differential equations (ODEs) given by
\be
\label{eq:autocatalator_2D}
\begin{array}{rcl}
\epsilon\dot{x}&=& yx^2+y-x,\\
        \dot{y}&=& \mu-yx^2-y,\\
\end{array}
\ee 
where $\dot{~}$ denotes derivative with respect to time, $x,y$ are dimensionless concentrations associated to $X,Y$ respectively and $\epsilon,\mu$ are parameters; we note that the parameter $\epsilon$ is defined by the ratio of reaction rates for $Y\ra X$ and $X\ra C$ \cite[p.6192]{PetrovScottShowalter}. Obviously it is natural to assume that the concentrations are non-negative $x,y\in \R^+_0:=\{w\in\R:w\geq 0\}$. Furthermore, note carefully that the nonlinear term arises due to the autocatalytic reaction part $Y+2X\ra 3X$. It has been proven in \cite{GucwaSzmolyan} that the 2D-autocatalator \eqref{eq:autocatalator_2D} can exhibit an attracting relaxation-oscillation periodic orbit for certain ranges of the parameters; see also Figure \ref{fig:1}(a). In \cite{PetrovScottShowalter} Petrov {et al.}~generalized \eqref{eq:reaction_2D} by including a further reactant
\be
\label{eq:reaction_3D}
P\ra Y,\quad  P+Z\ra Y+Z,\quad Y\ra X,\quad X\ra Z,\quad  Y+2X\ra 3X,\quad  Z\ra W.
\ee
As before, it is straightforward \cite{PetrovScottShowalter} to derive from \eqref{eq:reaction_3D} the ODEs  
\be
\label{eq:autocatalator_3D}
\begin{array}{rcl}
\epsilon\dot{x}&=& yx^2+y-x,\\
        \dot{y}&=& \mu(\kappa+z)-yx^2-y,\\
		    \dot{z}&=& x-z,\\
\end{array}
\ee 
where $(\mu,\kappa,\epsilon)$ are parameters and $(x,y,z)\in(\R^+_0)^3$ are the concentrations. Numerical studies \cite{GuckenheimerScheper,MilikSzmolyan,MilikSzmolyanLoeffelmannGroeller} have shown that periodic, mixed-mode and chaotic oscillations exist for the 3D autocatalator \eqref{eq:autocatalator_3D}.\medskip

\begin{figure}[htbp]
\psfrag{x}{$x$}
\psfrag{y}{$y$}
\psfrag{C0}{$\cC_0$}
\psfrag{a}{(a)}
\psfrag{b}{(b)}
	\centering
		\includegraphics[width=0.75\textwidth]{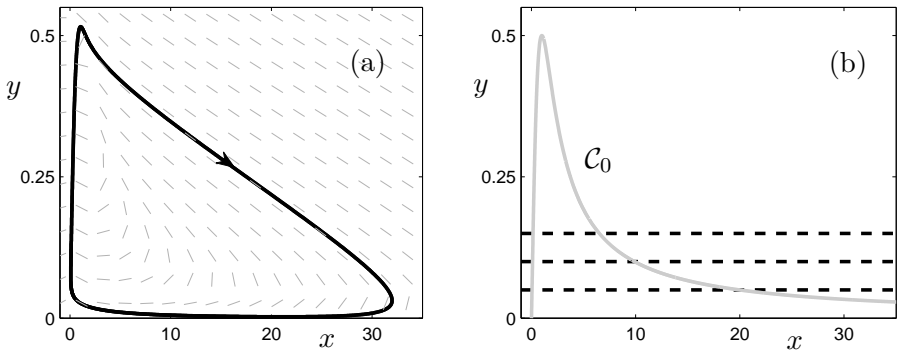}
	\caption{\label{fig:1}(a) Numerical illustration for the 2D autocatalator \ref{eq:autocatalator_2D} with $\mu=1.1$ and $\epsilon=0.01$. The thin line segments show the vector field. The attracting periodic relaxation-type periodic orbit (thick black curve) is also shown. (b) The critical manifold $\cC_0$ (gray) starts to align with the fast subsystem domains $\{y=\text{const.}\}$ (dashed) as $x\ra \I$.}	
\end{figure}

For both models \eqref{eq:autocatalator_2D}, \eqref{eq:autocatalator_3D} it is often assumed that the ratio of time scales $\epsilon>0$ is a sufficiently small parameter; we shall also denote this assumption by $0<\epsilon\ll1$. In this case it follows that both autocatalator models are fast-slow (or singularly perturbed) ODEs. We use the notation $p$ to denote a point in the phase space of concentrations for the autocatalator models {i.e.}~$p=(x,y)\in (\R_0^+)^2$ or $p=(x,y,z)\in(\R_0^+)^3$. The $x$-nullcline, or critical manifold, for both models is given by
\benn
\cC_0=\{p:yx^2+y-x=0\}=\left\{p:y=\frac{x}{1+x^2}\right\}.
\eenn
Note that $\cC_0$ is an unbounded smooth manifold. The angle between the tangent spaces $T_{p}\cC_0$ and the hyperplanes $\{p:y=\text{const.}\}$ decays to zero as $x\ra +\I$; see Figure \ref{fig:1}(b). This alignment, which is expected to imply a loss of normal hyperbolicity in the system, already causes substantial difficulties in the rigorous analysis of the dynamics of the (2D) autocatalator model \cite{GucwaSzmolyan}. The global return mechanism induced by the unbounded part of $\cC_0$ also plays a key role for the complex oscillatory patterns of the 3D autocatalator \cite[{p.~98}]{GuckenheimerScheper} which have been observed in numerical simulations.\medskip 

In a completely different context a model with similar properties to the autocatalator was proposed by Rankin et {al.}~\cite{RankinDesrochesKrauskopfLowenberg} as a caricature system for effects in aircraft ground dynamics 
\be
\label{eq:Rankin}
\left\{
\begin{array}{rcl}
\epsilon\dot{x}&=& y+(x-\mu)\exp\left(\frac{x}{\kappa}\right),\\
        \dot{y}&=& \nu-x,\\
\end{array}\right.
\ee
where $(\mu,\kappa,\nu,\epsilon)$ are parameters, $0<\epsilon\ll1$ and $\cC_0=\{y=(\mu-x)\exp\left(x/\kappa\right)\}$ also aligns with $\{y=0\}$ as $x\ra -\I$ if $\kappa>0$ or as $x\ra+\I$ if $\kappa<0$.\medskip

Motivated by these examples, it is desirable to build a general theory of fast-slow systems with unbounded critical manifolds. The main contributions of this work are:

\begin{itemize}
 \item We study a general class of critical manifolds which may have an arbitrary power-law decay for the alignment with a fast subsystem domain. This includes the autocatalytic critical manifolds as special cases and answers open questions arising from various numerical studies.
 \item Using the blow-up method we give a rigorous proof when normal hyperbolicity for a perturbation of the critical manifold fails. The relevant scaling law turns out to be given by 
\benn
(x,y)=\left(\cO(\epsilon^{-1/(s+1)}),\cO(\epsilon^{s/(s+1)})\right),\qquad  \text{as $\epsilon\ra 0$}
\eenn
where $s$ is a power-law exponent computable from the critical manifold. This scaling has immediate consequences for the asymptotics of oscillatory patterns.
 \item On a technical level we contribute to a further development of the blow-up method by augmenting it with an 'optimality-criterion' of blow-up coefficients which have to be chosen in the analysis.
\end{itemize}\medskip

The paper is structured as follows: In Section \ref{sec:background} we introduce the notation and the required background from fast-slow systems. In Section \ref{sec:asymptotics} a formal asymptotic argument is given to illustrate the important scaling properties and the system is localized via a projective transformation. The local system is desingularized via blow-up in Section \ref{sec:blowup}. The charts and transition functions between the two relevant charts are calculated. The dynamics in each chart are analyzed in Section \ref{sec:blowup1} and \ref{sec:blowup2} respectively. The optimality criterion for the blow-up coefficients is proven in Section \ref{sec:optimality}. The main scaling result is summarized in Section \ref{sec:main} and its implications are discussed.     

\section{Background and Notation}
\label{sec:background}

In this paper we restrict our attention to the analysis of family of planar fast-slow systems given by
\be
\label{eq:base_model}
\begin{array}{rcrcr}
\frac{dx}{dt}&=&x'&=& f(x,y),\\
\frac{dy}{dt}&=&y'&=& \epsilon g(x,y),\\
\end{array}
\ee
where $(x,y)\in\R^2$, $0<\epsilon\ll1$ and $f:\R^2\ra\R$, $g:\R^2\ra\R$ are assumed to be sufficiently smooth.\medskip

\textit{Remark:} By restricting to planar systems the (3D) autocatalator does not immediately fit within the theory we develop. However, we expect that center manifold techniques, similar to the generalization for folded singularities from three to arbitrary finite-dimensional fast-slow systems \cite{Wechselberger1}, can be applied.\medskip

We are going to choose particular forms of $f,g$ below but introduce some general terminology beforehand; for more detailed reviews/introductions to fast-slow systems see \cite{Jones,Kaper,Desrochesetal}. The critical set of \eqref{eq:base_model} is given by
\benn
\cC_0:=\left\{(x,y)\in\R^2:f(x,y)=0\right\}.
\eenn
We are going to assume that $\cC_0$ defines a smooth manifold. $\cC_0$ is called normally hyperbolic at a point $(x^*,y^*)\in\cC_0$ if $\partial_xf(x^*,y^*)\neq 0$. We will assume that $\cC_0$ is normally hyperbolic at every point in $\R^2$ but $\cC_0$ will be unbounded. Let us point out that there has been a lot of work on the geometric theory of planar fast-slow systems recently by De Maesschalck and Dumortier, see {e.g.} \cite{DeMaesschalckDumortier5,DeMaesschalckDumortier6,DeMaesschalckDumortier2}, particularly in the context of local singularities, periodic orbits and Li\'{e}nard systems.\medskip

\textit{Remark:} The general definition of normal hyperbolicity of an invariant manifold \cite{HirschPughShub,Fenichel1,WigginsIM} requires the splitting of tangent and normal dynamics. Recall that a compact manifold $\cM\subset \R^N$ is normally hyperbolic for a vector field $F:\R^N\ra\R^N$ if there exists a continuous splitting of the tangent bundle
\benn
T_{\cM}\R^n=N^u\oplus T\cM\oplus N^s 
\eenn
where the linearization $DF$ expands $N^u$ and contracts $N^s$ more sharply than $T\cM$. In $\R^2$, we obviously just need one of the normal bundles if $\dim(\cM)=1$. It is important to point out that we do restrict ourselves to perturbations of normally hyperbolic manifolds in the class of fast-slow vector fields of the form \eqref{eq:base_model} being interested in loss of normal hyperbolicity of invariant manifolds within this class of perturbed vector fields where the perturbation is given by $\epsilon g(x,y)$.\medskip

The implicit function theorem yields that $\cC_0$ is locally a graph. In our case, the parametrization will be global so that 
\benn
\cC_0=\{(x,y)\in\R^2:x=h(y)\}.
\eenn
for some smooth function $h:\R\ra \R$. The singular limit $\epsilon\ra0$ of \eqref{eq:base_model} defines a differential algebraic equation, also called the fast subsystem 
\be
\label{eq:base_model_fss}
\begin{array}{rcl}
x'&=& f(x,y),\\
y'&=& 0.\\
\end{array}
\ee
Changing to the slow time scale $\tau=\epsilon t$ in \eqref{eq:base_model} and taking the singular limit $\epsilon=0$ yields the slow subsystem defined on $\cC_0$
\be
\label{eq:base_model_ss}
\begin{array}{rcrcl}
0&=&0&=& f(x,y),\\
\frac{dy}{d\tau}&=&\dot{y}&=& g(x,y).\\
\end{array}
\ee
The slow subsystem can be written as $\dot{y}=g(h(y),y)$. If $h$ is invertible, the critical manifold can also be defined via $\cC_0=\{(x,y)\in \R^2:y=h^{-1}(x)=:c(x)\}$. Implicit differentiation with respect to $\tau$ yields $\dot{y}=c'(x)\dot{x}$ so that the slow subsystem is 
\benn
c'(x)\dot{x}= g(x,c(x)).
\eenn
Fenichel's Theorem \cite{Fenichel4,Jones} implies that that any compact submanifold $\cM_0\subset \cC_0$ persists as a slow manifold $\cM_\epsilon$ for $0<\epsilon\ll1$. Furthermore, $\cM_\epsilon$ is locally invariant and diffeomorphic to $\cM_0$. The flow on $\cM_\epsilon$ converges to the slow flow as $\epsilon\ra0$. For any fixed small $\epsilon>0$ we can define $\cM_\epsilon$ and then extend it under the flow of \eqref{eq:base_model} but this extension may no longer normally hyperbolic as an invariant manifold for the full system \eqref{eq:base_model}. To understand this effect we propose to study the family of model systems given by 
\be
\label{eq:poly_models}
f(x,y,\epsilon)=1-x^sy\qquad \text{for $s\in\N$}
\ee
and $g(x,y)=\mu$ for some $\mu\neq0$. The choice for the fast vector field is motivated by the asymptotic expansion of $\cC_0$ for the autocatalator models \eqref{eq:autocatalator_2D},\eqref{eq:autocatalator_3D} as $x\ra\I$. In fact, returning to the chemical interpretation the general model class corresponds to autocatalytic reaction steps of the form $Y+(s+1)X\ra (s+2)X$.\medskip 

The smooth unbounded critical manifold for \eqref{eq:poly_models} is given by
\be
\label{eq:C0_mod_class}
\cC_0:=\left\{(x,y)\in\R^2:y=\frac{1}{x^s}=c(x)\right\}.
\ee
As $|x|\ra\I$ the tangent space $T_{(x,y)}\cC_0$ starts to align with the $x$-axis since $c(x)\ra0$ and $c'(x)\ra0$ as $|x|\ra\I$. By reflection symmetry we restrict our focus to the positive part $\cC_0\cap(\R^+_0)^2$ from now on {i.e.}~the intersection with the non-negative quadrant will be understood. Note that the model class \eqref{eq:poly_models} covers already a wide variety of power-law decay rates. Since we are only interested in the unbounded part of $\cC_0$ we shall restrict the dynamics to 
\benn
\cH^\sigma:=\{(x,y)\in\R^2:x>\sigma\}
\eenn
for a suitable fixed $\sigma>0$ with $\sigma=\cO(1)$ as $\epsilon\ra 0$; see also Figure \ref{fig:2}(a). Since $\partial_xf(x^*,y^*)=-s(x^*)^{s-1}y^*\neq 0$ for $(x^*,y^*)\in\cC_0$ it follows that any compact submanifold $\cM_0\subset\cC_0$ is normally hyperbolic and perturbs, by Fenichel's Theorem, to a normally hyperbolic slow manifold for $\epsilon>0$ sufficiently small. However, for fixed small $\epsilon>0$ we do not know yet when the extension of $\cM_\epsilon$ under the flow starts to deviate from the critical manifold approximation. 

\section{Asymptotic Expansion and Projective Transformation}
\label{sec:asymptotics}

A good intuitive understanding of the problem is gained by considering one of the simplest possible systems in the model class \eqref{eq:poly_models} given by $s=1$ and $\mu=1$ so that
\be
\label{eq:example}
\begin{array}{rcl}
\epsilon \dot{x}&=& 1-yx,\\
         \dot{y}&=& 1.\\
\end{array}
\ee
From the viewpoint of formal asymptotics we can simply calculate a slow manifold $\cM_\epsilon$ by making the ansatz
\be
\label{eq:sm_ansatz}
x(y)=x_0(y)+\epsilon x_1(y)+\epsilon^2 x_2(y)+\cdots.
\ee
Inserting \eqref{eq:sm_ansatz} into \eqref{eq:example} and collecting terms of different orders in $\epsilon$ yields
\be
\label{eq:disorder}
\begin{array}{rrcl}
\cO(1):\qquad & 0&=&1-yx_0(y),\\
\cO(\epsilon^k):\qquad & \frac{dx_{k-1}}{dy}&=&-yx_k(y),\\
\end{array}
\ee 
for $k\in\N$. A direct calculation using \eqref{eq:disorder} gives the formal asymptotics of the slow manifold
\be
\label{eq:asy_series}
x(y)=\frac1y-\sum_{k=1}^K(2k-1)\frac{\epsilon^k}{y^{2k+1}}+\cO(\epsilon^{K+1})\qquad \text{as $\epsilon\ra0$}.
\ee
The series \eqref{eq:asy_series} fails to be asymptotic when $1/y\sim \epsilon^k/y^{2k+1}$ as $\epsilon\ra0$ or $y=\cO(\sqrt\epsilon)$. Since the critical manifold of \eqref{eq:example} is given as the graph of $y=1/x$ we expect that normal hyperbolicity of the associated slow manifold breaks down for $x=\cO(\epsilon^{-1/2})$. The example \eqref{eq:example} suggests that, depending on the asymptotics of the critical manifold as $x\ra \I$, there is an $\epsilon$-dependent scaling law that governs within which regime we can use the critical manifold approximation provided by Fenichel's Theorem. To give a rigorous proof of the asymptotic scaling we would like to desingularize the problem. In \cite{GucwaSzmolyan} the autocatalator is considered and phase space is re-scaled by an $\epsilon$-dependent transformation; see also further remarks in Section \ref{sec:main}. Here we use a different approach by first completely localizing the problem; see also \cite{Wechselberger3}. Consider a projective coordinate transformation 
\be
\label{eq:projective}
\rho:\cH^1\ra(\R^2-\cH^1)\qquad \rho(x,y)=\left(\frac{1}{v},y\right)
\ee
where $\sigma=1$ is chosen for computational convenience for the domain $\cH^\sigma$. All the following results are easily checked to be independent of a fixed chosen $\sigma>0$ independent of $\epsilon$; see also Figure \ref{fig:2}(b). A direct calculation yields the following result:\medskip 

\begin{lem}
The map \eqref{eq:projective} yields a transformed vector field defined on $(\R^2-\cH^1)$ and given by
\be
\label{eq:main_local}
\begin{array}{rcl}
\epsilon\dot{v}&=& \left(yv^{2-s}-v^{2}\right),\\
\dot{y}&=& \mu. 
\end{array}
\ee
\end{lem}

\begin{figure}[htbp]
\psfrag{x}{$x$}
\psfrag{y}{$y$}
\psfrag{s=1}{$s=1$}
\psfrag{s=2}{$s=2$}
\psfrag{rho}{$\rho$}
\psfrag{v}{$v$}
\psfrag{H2}{$\R^2-\cH^1$}
\psfrag{H1}{$\cH^1$}
\psfrag{a}{(a)}
\psfrag{b}{(b)}
	\centering
		\includegraphics[width=0.75\textwidth]{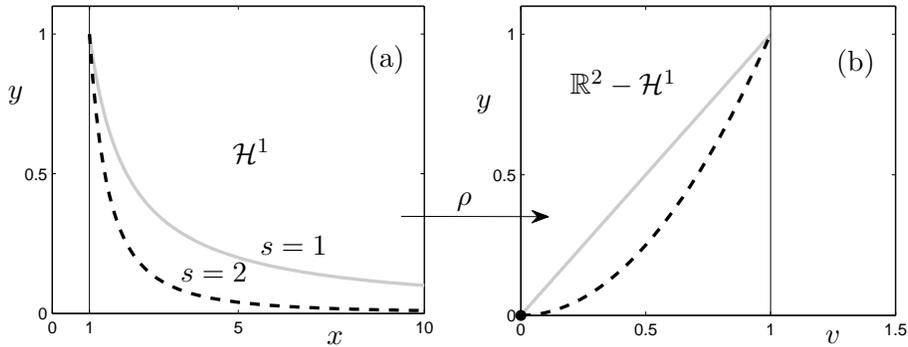}
	\caption{\label{fig:2}(a) Critical manifolds \eqref{eq:C0_mod_class} for $s=1$ (solid gray) and $s=2$ (dashed black). (b) Critical manifolds after the transformation \eqref{eq:projective}. Recall that we always restrict the dynamics to the non-negative quadrant $(\R^+_0)^2$.}	
\end{figure}

The vector field \eqref{eq:main_local} can be desingularized, similar to the desingularization for folded singularities \cite{SzmolyanWechselberger1}, via multiplication by $v^s$ and a time re-scaling which leads to 
\be
\label{eq:main_local1}
\begin{array}{rcl}
\epsilon\dot{v}&=& v^2\left(y-v^{s}\right),\\
\dot{y}&=& \mu v^{s}. 
\end{array}
\ee
Note that the transformation reverses time on orbits for $v<0$ if $s$ is odd but that we always restrict to the non-negative quadrant which implies that we do not have to consider this issue here. After the transformation the original critical manifold $\cC_0=\{y=1/x^s\}$ is given by 
\benn
\rho(\cC_0)=\left\{(v,y)\in\R^2-\cH^1:y=v^{s}\right\}=:\cS_0.
\eenn
Observe that $\cS_0$ is normally hyperbolic except at the degenerate singularity $(v,y)=(0,0)$. The key point of the projective transformation $\rho$ is that it localized the approach towards a singular point to calculate the breakdown of normal hyperbolicity; in fact, it is not the singularity arising from infinity we are interested in but the scaling of the associated slow manifold $\cS_\epsilon$ in a sufficiently small ball excluding $(v,y)=(0,0)$. Note that the case $s=1$ yields the classical transcritical structure for $\cS_0$ with $f(v,y)=v(y-v)$; see \cite{KruSzm4,Schecter} and references therein. To desingularize \eqref{eq:main_local} one can try to apply a blow-up transformation. The blow-up method \cite{Dumortier1} was introduced into fast-slow systems theory in the work of Dumortier and Roussarie \cite{DumortierRoussarie}. For additional background on the blow-up method for fast-slow local singularities see \cite{KruSzm1}. 

\section{Desingularization via Blow-Up}
\label{sec:blowup}

Re-writing \eqref{eq:main_local1} on the fast time scale and augmenting the system by $\epsilon'=0$ leads to the vector field 
\be
\label{eq:main_local2}
\begin{array}{rcl}
v'&=& v^2\left(y-v^{s}\right),\\
y'&=& \epsilon \mu v^{s},\\ 
\epsilon'&=&0, 
\end{array}
\ee
which we denote by $X$. Consider the manifold $\bar{\cB}:=\cS^2\times [0,r_0]$ for some $r_0>0$ where $\cS^2\subset \R^3$ denotes the unit sphere. Denote the coordinates on $\bar{\cB}$ by $(\bar{v},\bar{y},\bar\epsilon,\bar{r})$ where $(\bar{v},\bar{y},\bar\epsilon)\in\cS^2$. A general weighted blow-up transformation $\Phi:\bar{\cB}\ra \R^3$ is then given by
\be
\label{eq:bu1choice}
v=\bar{r}^{\alpha_v}\bar{v},\qquad y=\bar{r}^{\alpha_y}\bar{y},\qquad \epsilon=\bar{r}^{\alpha_\epsilon}\bar{\epsilon},  
\ee 
where we choose the coefficients as
\be
\label{eq:bu2choice}
(\alpha_v,\alpha_y,\alpha_\epsilon)=(1,s,s+1). 
\ee
The choice of $(\alpha_v,\alpha_y,\alpha_\epsilon)$ can be motivated by using the Newton polygon \cite{Dumortier1,BerglundKunz}. Via $\Phi$ a vector field $\bar{X}$ is induced on the space $\bar{\cB}$ by requiring $\Phi_*(\bar{X})=X$ where $\Phi_*$ is the usual push-forward. Recall that we restrict to the manifold $\cC_0$ in the positive quadrant. Let $\cB_{\bar{y}}:=\bar{\cB}\cap\{\bar{y}>0\}$ and $\cB_{\bar{\epsilon}}:=\bar{\cB}\cap\{\bar{\epsilon}>0\}$ and consider two charts 
\benn
\kappa_1:\cB_{\bar{y}}\ra \R^3 \qquad \text{and}\qquad \kappa_2:\cB_{\bar{\epsilon}}\ra \R^3
\eenn
where $(v_1,r_1,\epsilon_1)$ and $(v_2,y_2,r_2)$ denote the respective coordinates for $\cB_{\bar{y}}$ and $\cB_{\bar{\epsilon}}$; see also \cite[Fig2]{KruSzm1} for the basic geometry of a blow-up in $\R^3$. Slightly tedious, but straightforward, calculations yield the following result:\medskip

\begin{lem}
\label{lem:kappa_maps}
The maps $\kappa_1$ and $\kappa_2$ are given by
\benn
\begin{array}{lll}
v_1=\bar{v}\bar{y}^{-1/s},\quad & r_1=\bar{r}\bar{y}^{1/s}, \quad & \epsilon_1=\bar{\epsilon}\bar{y}^{-(s+1)/s},\\
v_2=\bar{v}\bar{\epsilon}^{-1/(s+1)},\quad &  y_2=\bar{y}\bar{\epsilon}^{-s/(s+1)},\quad & r_2=\bar{r}\bar{\epsilon}^{1/(s+1)}.\\
\end{array}
\eenn
The transitions functions $\kappa_{12}$ and $\kappa_{21}$ between charts are
\benn
\begin{array}{lll}
v_2=v_1\epsilon_1^{-1/(s+1)}, \quad & y_2=\epsilon_1^{-s/(s+1)},\quad & r_2=r_1\epsilon_1^{1/(s+1)},\\
v_1=v_2y_2^{-1/s},\qquad & r_1=r_2y_2^{1/s},\quad & \epsilon_1=y_2^{-(s+1)/s}.
\end{array}
\eenn
\end{lem}\medskip

\begin{lem}
\label{lem:kappa_flows}
The desingularized vector field in $\kappa_1$ is given by
\be
\label{eq:vf_kappa1}
\begin{array}{lcl}
v_1'&=&sv_1^2(1-v_1^s)-\mu\epsilon_1 v_1^{s+1},\\
r_1'&=&\mu r_1\epsilon_1v_1^s,\\
\epsilon_1'&=&-(s+1)\mu\epsilon_1^2v_1^s,\\
\end{array}
\ee
The desingularized vector field in $\kappa_2$ is given by
\be
\label{eq:vf_kappa2} 
\begin{array}{lcl}
v_2'&=&v_2^2(y_2-v_1^s),\\
y_2'&=&\mu v_2^s.\\
\end{array}
\ee
\end{lem}

\begin{proof}(of Lemma \ref{lem:kappa_flows})
The required calculations follow by direct differentiation, for example, one has
\benn
r_1'=(\mu/s) r_1^{s+2}\epsilon_1v_1^s\qquad  \text{and} \qquad v'=r_1'v_1+r_1v_1' 
\eenn
which imply upon algebraic manipulation that
\benn
v_1'=r_1^{s+1}\left(v_1^2(1-v_1^s)-\frac{\mu}{s}\epsilon_1 v_1^{s+1}\right).
\eenn
Similarly, one obtains the equation for $\epsilon_1$ given by 
\benn
\epsilon_1'=-\frac{(s+1)}{s}\mu r^{s+1}\epsilon_1^2v_1^s.
\eenn
A division by the common factor $r_1^{s+1}$ and a time re-scaling yield \eqref{eq:vf_kappa1}. In $\kappa_2$ the transformation is just a re-scaling $(v,y)=\left(\epsilon^{1/(s+1)}v_2,\epsilon^{s/(s+1)}y_2\right)$ and a time rescaling $t\ra t/\epsilon$ gives \eqref{eq:vf_kappa2}.  
\end{proof}

\section{Dynamics in the First Chart}
\label{sec:blowup1}

We can continue the critical manifold $\cS_0=\{(v,y)\in(\R^2)^+:y=v^s\}$ into the chart $\kappa_1$. Using the definition of the blow-up map and Lemma \eqref{lem:kappa_maps} imply the relevant continuation result:\medskip

\begin{lem}
$(\kappa_1\circ\Phi^{-1})(\cS_0)=\{(v_1,y_1)\in(\R^2)^+:v_1=1\}=:\cS_{1,a}$.
\end{lem}\medskip

Analyzing the flow \eqref{eq:vf_kappa1} gives more information about $\cS_{1,a}$. The vector field \eqref{eq:vf_kappa1} has two invariant subspaces $\epsilon_1=0$ and $r_1=0$. In the first subspace the flow is 
\be
\label{eq:flow11}
\begin{array}{lcl}
v_1'&=&sv_1^2(1-v_1^s),\\ 
r_1'&=&0,\\
\end{array}
\ee 
which has a line of equilibrium points $\cS_{1,a}$ which are attracting in the $v_1$-direction. There is also a line of equilibria at $v_1=0$ which we do not have to consider. In the second invariant subspace the flow is
\be
\label{eq:flow12}
\begin{array}{lcl}
v_1'&=&sv_1^2(1-v_1^s)-\mu\epsilon_1 v_1^{s+1},\\
\epsilon_1'&=&-(s+1)\mu\epsilon_1^2v_1^s.\\
\end{array}
\ee
Linearizing around $(v_1,\epsilon_1)=(1,0)$ provides the important local dynamics:\medskip

\begin{lem}
\label{lem:eqpt}
The vector field \eqref{eq:flow12} has a center-stable equilibrium at $(v_1,\epsilon_1)=(1,0)=:p_{1,a}$ with eigenvalues $0$ and $-s^2$ and associated eigenvectors $(-\mu/s^2,1)^T$ and $(1,0)^T$.
\end{lem}\medskip

\begin{lem}
\label{lem:cm}
Locally near $p_{a,1}$ the equilibrium $p_{1,a}$ has a one-dimensional center manifold
\benn
 \cN_{1,a}=\{v_1=1-\epsilon_1\frac{\mu}{s}+c_{11}\epsilon_1^2+\cO(\epsilon_1^3):=n_{1,a}(\epsilon_1)+\cO(\epsilon_1^3))\}
\eenn
where $c_{11}=-(1+s+s^2)\mu^2/(2s^3)$.
\end{lem}

\begin{proof}(of Lemma \ref{lem:cm})
Existence is guaranteed by the center manifold theorem \cite{GH}. Using a translation $V_1=v_1-1$ moves $p_{1,a}$ to the origin. A further linear coordinate change 
\benn
\left(\begin{array}{c}V_1 \\ \epsilon_1\\ \end{array}\right)=Mz=\left(\begin{array}{cc}-\mu/s^2 & 1 \\ 1 & 0\\ \end{array}\right)
\left(\begin{array}{c}z_1 \\ z_2\\ \end{array}\right)
\eenn
transforms \eqref{eq:flow12} into
\benn
\begin{array}{lcrll}
z_1'&=& 0 &-\mu(1+s)z_1^2&+\cO(3),\\
z_2'&=& -s^2 &-\mu^2\left(1+\frac{1}{2s}+\frac{1}{2s^2}\right)z_1^2+2\mu z_1z_2-\left(\frac{3s^2}{2}+\frac{s^3}{2}\right)z_2^2&+\cO(3),\\
\end{array}
\eenn
where $\cO(3)=\cO(z_1^3,z_1^2z_2,z_1z_2^2,z_2^3)$. Making the ansatz $z_2=c_{11}z_1^2+\cO(z_1^3)$ and substituting into the invariance equation \cite{GH} for the center manifold at $z=(0.0)$ implies the condition
\benn
0=\left(c_{11}s+\mu^2+\frac{\mu^2}{2s}+\frac{\mu^2}{2s^2}\right)z_1^2+\cO(z_1^3).
\eenn
Therefore, $c_{11}=-(1+s+s^2)\mu^2/(2s^3)$ and transforming back to $(v_1,\epsilon_1)$ yields the result.
\end{proof}\medskip

As before, other equilibria of \eqref{eq:flow12} will not be of relevance here. Lemma \ref{lem:eqpt} and the center manifold theorem \cite{GH} imply the next result:\medskip

\begin{lem}
\label{lem:extend}
There exists a center-stable manifold $\cM_{1,a}$ for \eqref{eq:vf_kappa1} at $p_{1,a}$ containing $\cS_{1,a}$ and $\cN_{1,a}$. Furthermore, $\cM_{1,a}$ is locally given as a graph of a map $y_1=h_{1,a}(r_1,\epsilon_1)$.  
\end{lem}\medskip

From the flow on the center manifold we see that there are two cases based upon the sign of $\mu$. If $\mu>0$ then trajectories in the center manifold flow away from the sphere $\cS^2$ while for $\mu<0$, trajectories in $\cM_{1,a}$ flow from the chart $\kappa_1$ onto the sphere. We are only going to deal with the case $\mu<0$ from now on. The case $\mu>0$ can be obtained from $\mu<0$ by a time reversal of the the original problem \eqref{eq:main_local2}. Figure \ref{fig:3}(a) provides a sketch of the results and the notation for the relevant dynamical objects.

\begin{figure}[htbp]
\psfrag{v1}{$v_1$}
\psfrag{e1}{$\epsilon_1$}
\psfrag{r1}{$r_1$}
\psfrag{v2}{$v_2$}
\psfrag{y2}{$y_2$}
\psfrag{M1}{$\cM_{1,a}$}
\psfrag{N1}{$\cN_{1,a}$}
\psfrag{N2}{$\cN_{2,a}$}
\psfrag{n2a}{$n_{2,a}(t)$}
\psfrag{p1}{$p_{1,a}$}
\psfrag{S1}{$\cS_{1,a}$}
\psfrag{a}{(a)}
\psfrag{b}{(b)}
	\centering
		\includegraphics[width=0.6\textwidth]{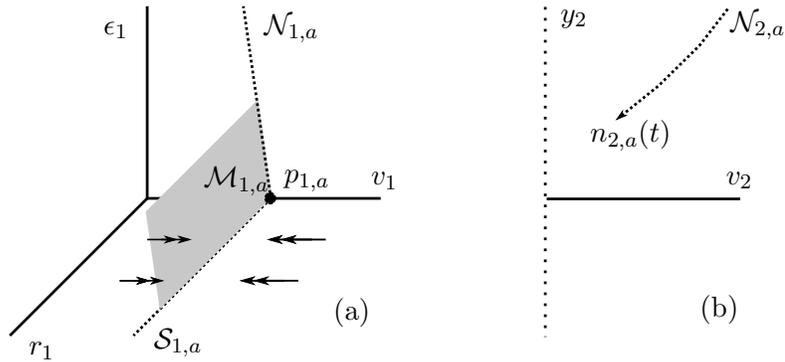}
	\caption{\label{fig:3}(a) Sketch of the dynamics in chart $\kappa_1$. (b) Sketch of the dynamics in chart $\kappa_2$}	
\end{figure}

\section{Dynamics in the Second Chart}
\label{sec:blowup2}

Now the critical manifold in the chart $\kappa_1$ has to be continued into the chart $\kappa_2$. This transition is studied using the continuation $\cM_{2,a}$ of $\cM_{1,a}$ under the flow in $\kappa_2$; see also Figure \ref{fig:3}(b). \medskip

\begin{lem}The curve $\cN_{1,a}\subset \cM_{1,a}$ transforms as 
\benn
\kappa_{12}(\cN_{1,a})=\left\{(v_2,y_2)\in\R^2:v_2=1-\frac{\mu}{s y_2}+\frac{c_{11}}{y_2^{(2s+1)/s}}+\cO\left(y_2^{\frac{-3s-2}{s}}\right)\text{ as $y_2\ra +\I$}\right\}.
\eenn
\end{lem}\medskip

We introduce the notation $\cN_{2,a}:=\kappa_{12}(\cN_{1,a})$ and note that $\cN_{2,a}$ can be extended under the flow \eqref{eq:vf_kappa2} in chart $\kappa_2$ for initial conditions with $v_2\approx 1$ and arbitrarily large $y_2>0$. To emphasize that this extension is a trajectory we are going to denote it by $n_{2,a}(t)=(n_{v_2,a}(t),n_{y_2,a}(t))^T\in\R^2$. A case distinction based upon the value of $s$ will be necessary. For $s=2k$ with $k\in\N$ the differential equation is    
\be
\label{eq:vfs2} 
\begin{array}{lcl}
v_2'&=&v_2^2(y_2-v_2^{2k}),\\
y_2'&=&\mu v^{2k}_2,\\
\end{array}
\ee
so that $y_2'<0$ and $v_2'<0$ for $y_2<v_2^{2k}$ imply that $n_{v_2,a}(t)\ra 0$ for $t\ra +\I$. The variational equation is defined via the time-dependent matrix
\beann
A(t)&:=&\left.D_{(v_2,y_2)}\left(\begin{array}{c}v_2^2(y_2-v_2^{2k}) \\ \mu v^{2k} \\\end{array}\right)\right|_{n_{2,a}(t)}\\
&=&\left(\begin{array}{cc}2n_{y_2,a}(t)n_{v_2,a}(t)-(2k+2)n_{v_2,a}(t)^{2k+1} & 1 \\ 2k \mu ~n_{v_2,a}(t)^{2k-1} & 0 \\\end{array}\right),
\eeann 
which shows that the variational equation becomes asymptotically autonomous in forward time \cite{Rasmussen} and $A(t)$ approaches a double zero eigenvalue when $t\ra +\I$. Hence $\cM_{2,a}$ is not everywhere normally hyperbolic in $\kappa_2$. A possible alternative argument for $k>1$, leading to the same conclusion, involves desingularizing \eqref{eq:vfs2} using division by $v_2^2$ and observing that the resulting vector field is parallel on the $y_2$-axis with no flow in the $y_2$-component. However, with this argument the case $k=0$ is special but can be treated using standard results about the resulting Ricatti equation which also appears in the blow-up analysis of the fold point \cite{KruSzm3}.\medskip 

For $s=2k+1$ and $k\in\N_0$ the vector field is 
\be
\label{eq:vfs3} 
\begin{array}{lcl}
v_2'&=&v_2^2(y_2-v_2^{2k+1}),\\
y_2'&=&\mu v^{2k+1}_2.\\
\end{array}
\ee
Note that \eqref{eq:vfs3} does not have the same nice monotonicity properties as before. To show the convergence towards $v_2=0$ define 
\benn
V(v_2,y_2):=\frac{1}{2k}v_2^{2k}+\frac{1}{2|\mu|}y_2^{2}\quad \Rightarrow\quad \frac{dV}{dt}=v_2^{2k-1}v_2'+\frac{1}{|\mu|}y_2y_2'=-v_2^{2(2k+1)}<0
\eenn 
showing that $V$ is a Lyapunov function \cite{HirschSmaleDevaney} which implies the required convergence. Computing the variational equation yields
\beann
A(t)&:=&\left.D_{(v_2,y_2)}\left(\begin{array}{c}v_2^2(y_2-v_2^{2k+1}) \\ \mu v_2^{2k+1} \\\end{array}\right)\right|_{n_{2,a}(t)}\\
&=&\left(\begin{array}{cc}2n_{v_2,a}(t)n_{y_2,a}(t)-(2k+3)n_{v_2,a}(t)^{2k+2} & n_{v_2,a}(t)^2 \\ \mu(2k+1)n_{v_2,a}(t)^{2k} & 0 \\\end{array}\right).
\eeann 
which shows that the variational equation becomes asymptotically autonomous in forward time and $A(t)$ approaches a double zero eigenvalue when $t\ra +\I$. As before, one could also use a suitable division by a power of $v_2$ and observe that the resulting vector field is parallel on the $y_2$-axis.\medskip

\begin{lem}
\label{lem:not_hyp}
$\cM_{2,a}$ is not everywhere a normally hyperbolic manifold inside $\kappa_2$.
\end{lem}\medskip

Lemma \ref{lem:not_hyp} is just a local restatement of the alignment property of the original slow manifold $\cC_\epsilon$ with the fast subsystem domains. This confirms the conjectured loss of normal hyperbolicity as $x\ra+\I$. The previous results combine to give the following statement which is a partial version of our main result.\medskip

\begin{prop}
\label{prop:main1}
Consider the family of planar fast-slow systems
\be
\label{eq:base_model_fss_main} 
\begin{array}{rcl}
x'&=& 1-x^sy,\\
y'&=& \epsilon \mu.\\
\end{array}
\ee
for $s\in\N$, $\mu\neq0$ and $\epsilon>0$ sufficiently small. Let $\cM_0\subset \cC_0$ be a compact submanifold of the critical manifold. Then $\cM_\epsilon$ extends to a normally hyperbolic manifold of \eqref{eq:base_model_fss_main} up to a domain of size 
\benn
(x,y)=\left(\cO(\epsilon^{-1/(s+1)}),\cO(\epsilon^{s/(s+1)})\right),\qquad  \text{as $\epsilon\ra 0$.}
\eenn
Under the blow-up map \eqref{eq:bu1choice}-\eqref{eq:bu2choice} the manifold $\cM_\epsilon$ cannot be extended to a normally hyperbolic slow manifold to a subset of a larger domain. 
\end{prop}

\begin{proof}
We work in projective coordinates via $\rho$. We can restrict to pieces of the critical manifold lying in the positive quadrant without loss of generality. $\cM_\epsilon$ is obtained from $\cM_0$ by Fenichel's Theorem since $\epsilon>0$ is sufficiently small. Using the blow-up transformation and Lemma \ref{lem:extend}, it follows that $\cM_\epsilon$ extends to a normally hyperbolic manifold in the chart $\kappa_1$. By Lemma \ref{lem:not_hyp} the extension is not normally hyperbolic in the chart $\kappa_2$. By Lemma \ref{lem:kappa_maps} and Lemma \ref{lem:kappa_flows} it follows that the neighborhood of $(v,y)=(0,0)$ scales as 
\be
\label{eq:sregion}
(v,y)=\left(\epsilon^{1/(s+1)}v_2,\epsilon^{s/(s+1)}y_2\right).
\ee
By applying a blow-down transformation it follows that $\cM_\epsilon$ is normally hyperbolic up to a scaling region given by \eqref{eq:sregion} and not normally hyperbolic inside some larger domain. Via $\rho^{-1}$ we get $(x,y)=\left(\cO(\epsilon^{-1/(s+1)}),\cO(\epsilon^{s/(s+1)})\right)$ as $\epsilon\ra0$ in original coordinates so that the result follows. 
\end{proof}

\section{Smaller Regions}
\label{sec:optimality}

To prove the final result, we aim to strengthen Proposition \ref{prop:main1} as it claims that the slow manifold extends and then normal hyperbolicity breaks down for the scaling \eqref{eq:sregion} using the blow-up \eqref{eq:bu1choice}-\eqref{eq:bu2choice}. A priori, we could have chosen different exponents for the blow-up which could have allowed us to extend the slow manifold even further. Here we show that the exponents in \eqref{eq:sregion} are indeed optimal in the sense that for any pair $(\alpha_1,\alpha_2)$, $\alpha_i\geq0$ with $0<\alpha_1+\alpha_2$ the manifold $\cC_\epsilon$ is not normally hyperbolic in the larger region
\benn
(x,y)=\left(\cO(\epsilon^{-(1+\alpha_1)/(s+1)}),\cO(\epsilon^{(s+\alpha_2)/(s+1)})\right),\qquad  \text{as $\epsilon\ra 0$.}
\eenn
Note that it suffices to assume that $0<\alpha_1+\alpha_2\ll1$ so that the region is slightly larger. Consider the modified blow-up 
\be
\label{eq:second_bu}
v=\bar{r}^{1+\alpha_1}\bar{v},\qquad y=\bar{r}^{s+\alpha_2}\bar{y},\qquad \epsilon=\bar{r}^{s+1}\bar{\epsilon}.
\ee
Observe that in the chart $\kappa_2$ the blow-up \eqref{eq:second_bu} reduces to the re-scaling
\benn
v=\epsilon^{(1+\alpha_1)/(s+1)}v_2,\qquad y=\epsilon^{(s+\alpha_2)/(s+1)}y_2
\eenn
which yields a strictly smaller scaling region than before. \medskip

\begin{lem}
In the chart $\kappa_1$ the vector field is 
\be
\label{eq:2nd_bu_k1}
\begin{array}{lcl}
r_1'&=&\frac{\mu}{s+\alpha_2}r_1^{2-\alpha_2+s+s\alpha_1}\epsilon_1v_1^s,\\
v_1'&=& \frac{-\mu(1+\alpha_1)}{s+\alpha_2}r_1^{1-\alpha_2+s+s\alpha_1}\epsilon_1v_1^{s+1}+r_1^{1+\alpha_1+s}v_1^2(r_1^{\alpha_2}-r_1^{s\alpha_1}v_1^s),\\
\epsilon_1&=& \frac{-(s+1)\mu}{s+\alpha_2}r_1^{1-\alpha_2+s+s\alpha_1}\epsilon_1^2v_1^s,\\
\end{array}
\ee
\end{lem}\medskip

\begin{prop}
\label{prop:main2}
The extension of the manifold $\cS_\epsilon$ in the chart $\kappa_1$ for the blow-up \eqref{eq:second_bu} is not normally hyperbolic as $r_1\ra0$. 
\end{prop}

\begin{proof}
Suppose first that $s\alpha_1-\alpha_2\leq 0$. Since $0<\alpha_1+\alpha_2\ll1$ we can desingularize the vector field \eqref{eq:2nd_bu_k1} by $r_1^{1-\alpha_2+s+s\alpha_1}$ as $1-\alpha_2+s+s\alpha_1>0$. Considering the invariant subspace defined by $\epsilon_1=0$ of the desingularized vector field yields the equation 
\be
\label{eq:sapp1}
v_1'= r_1^{\alpha_1+\alpha_2-s\alpha_1}v_1^2(r_1^{\alpha_2}-r_1^{s\alpha_1}v_1^s)
\ee
which has a line of equilibria given by $v_1=r_1^{(\alpha_2-s\alpha_1)/s}$ corresponding to the critical manifold $\cS_0$. Linearizing \eqref{eq:sapp1} around this line gives the variational equation
\benn 
V'=r_1^{\alpha_1+\alpha_2-s\alpha_1}(2r_1^{(\alpha_2-s\alpha_1)/s}r_1^{\alpha_2}-(s+1)r_1^{s\alpha_1}r_1^{(\alpha_2-s\alpha_1)(s+1)/s})V
\eenn
which reduces to $V'=0\cdot V$ as $r_1\ra0$. Hence, in this case, the extension of $\cS_\epsilon$ is not normally hyperbolic in all of $\kappa_1$. In the case $s\alpha_1-\alpha_2>0$ we can desingularize by $r_1^{\alpha^*}$ where $\alpha^*=1+s+\min(\alpha_1+\alpha_2,s\alpha_1-\alpha_2)$ in which case again a multiplicative factor containing $r_1$ appears in the variational equation.
\end{proof}\medskip

The main point of the previous proof is that we cannot remove the $r_1$-dependence in the the chart $\kappa_1$ which means that there is no center manifold at the point $p_{1,a}$ extending up to the sphere $\cS^2$, as previously for the blow-up with $\alpha_1=0=\alpha_2$.

\section{The Main Result and Conclusions}
\label{sec:main}

Finally, we can combine the previous result to obtain the main result about loss of normal hyperbolicity and scaling. Recall again that we always restrict the dynamics to $(\R_0^+)^2\cap \cH^1$ which is a subset in the non-negative quadrant bounded away from $x=0$.\medskip 

\begin{thm}
\label{thm:main}
Consider the family of planar fast-slow systems
\be
\label{eq:base_model_fss_main_final} 
\begin{array}{rcl}
x'&=& 1-x^sy,\\
y'&=& \epsilon \mu.\\
\end{array}
\ee
for $s\in\N$, $\mu\neq0$ and $\epsilon>0$ sufficiently small. Let $\cM_0\subset \cC_0$ be a compact submanifold of the critical manifold. Then $\cM_\epsilon$ extends to a normally hyperbolic manifold of \eqref{eq:base_model_fss_main_final} up to a domain of size 
\be
\label{eq:main_scale}
(x,y)=\left(\cO(\epsilon^{-1/(s+1)}),\cO(\epsilon^{s/(s+1)})\right),\qquad  \text{as $\epsilon\ra 0$.}
\ee
under the flow of \eqref{eq:base_model_fss_main_final}. $\cM_\epsilon$ is not normally hyperbolic for any larger domain {i.e.}~there does not exist a normally hyperbolic invariant manifold $\tilde{\cM}_\epsilon$ in a domain strictly larger than \eqref{eq:main_scale} as $\epsilon \ra 0$ such that $\cM_\epsilon$ is strictly contained in $\tilde{\cM}_\epsilon$.  
\end{thm}

\begin{proof}
Apply Proposition \ref{prop:main1} and Proposition \ref{prop:main2}. 
\end{proof}\medskip

Theorem \ref{thm:main} immediately applies to the 2D autocatalator. It is very interesting to point out the relation between the approach by Gucwa and Szmolyan \cite{GucwaSzmolyan} and the general Theorem \ref{thm:main}. The scaling used in \cite{GucwaSzmolyan} to capture the large relaxation-type oscillation for the autocatalator is $x=\tilde{x}/\epsilon$. Theorem \ref{thm:main} states that the slow manifold loses normal hyperbolicity when $x=\cO(\epsilon^{-1/2})$. This scaling certainly does not exclude the possibility of global crossings of the critical manifold when $x=\cO(\epsilon^{-1})$. In particular, the relaxation loop shown in Figure \ref{fig:1}(a) jumps near a fold point of the critical manifold 'towards infinity' but it has a travel time so that it only turns around for $x=\cO(\epsilon^{-1})$. This yields one possible explanation why two consecutive blow-ups are needed in \cite{GucwaSzmolyan} as the scaling $x=\tilde{x}/\epsilon$ compresses two important scaling regions for the global return simultaneously.\medskip

For the model \eqref{eq:Rankin} by Rankin et {al.}~the decay rate of the critical manifold is exponential and hence faster than any polynomial. In this context, Theorem \ref{thm:main} implies that a compact submanifold of the critical manifold will loose normal hyperbolicity for arbitrarily large choices of $s\in\N$ in \eqref{eq:main_scale}. Hence the scaling is expected to be $x,y=\cO(1)$, or at least close to this scaling for all practical purposes, so that normal hyperbolicity is lost in a very small region independent of $\epsilon$.\medskip

As an important conclusion from Theorem \ref{thm:main} one may also understand the behaviour of large amplitude oscillations (LAOs) induced by a global return mechanism with the asymptotic dynamics given by \eqref{eq:base_model_fss_main_final} {i.e.}~those LAOs which cross the critical manifold. For example, given a general autocatalytic reaction mechanism $Y+(s+1)X\ra (s+2)X$ then the lower bound for the crossing is $\cO(\epsilon^{-1/(s+1)})$ for the fast component amplitude. Since $s$ is usually known and $\epsilon$ can be computed from the reaction rates the result has quite general applicability.\medskip 

\textbf{Acknowledgments:} I would like to thank the European Commission (EC/REA) for support by a Marie-Curie International Re-integration Grant. I also acknowledge support via an APART fellowship of the Austrian Academy of Sciences ({\"{O}AW}).

\end{document}